\documentclass[twoside]{amsart}
\usepackage{amsmath,amsfonts,amsthm,mathrsfs}
\usepackage[unicode]{hyperref}
\usepackage[alphabetic,initials]{amsrefs}
\usepackage{amssymb}
\usepackage[dpi=1270,PostScript=dvips,heads=LaTeX]{diagrams}  
\diagramstyle{small,labelstyle=\scriptstyle}

\newcommand{\ltriangle}[4][]%
{\begin{diagram}[#1]%
	{#2} &\rTo& {#3} &\rTo& {#4} &\rTo& {#2 \decal{1}}%
\end{diagram}}

\usepackage{chngcntr}

\newcommand{\tensor}{\otimes}

\newcommand{\Hom}[3][]{\operatorname{Hom}_{#1}(#2,#3)}

\newcommand{\ZZ}{\mathbb{Z}}
\newcommand{\QQ}{\mathbb{Q}}

\newcommand{\HH}{\mathbb{H}}

\newcommand{\PPn}[1]{\mathbb{P}^{#1}}

\newtheorem*{lem*}{Lemma}

\newtheorem*{thm*}{Theorem}
\newtheorem*{prop*}{Proposition}

\theoremstyle{definition}

\theoremstyle{remark}

\newtheorem*{ex*}{Example}

\newtheorem*{note}{Note}

\theoremstyle{plain}

\DeclareMathOperator{\rk}{rk}

\DeclareMathOperator{\Spec}{Spec}

\DeclareMathOperator{\Supp}{Supp}

\newcommand{\define}[1]{\emph{#1}}

\newcommand{\OB}{\mathcal{O}_B}

\newcommand{\shf}[1]{\mathscr{#1}}
\newcommand{\OX}{\shf{O}_X}

\newcommand{\argbl}{\,\_\!\_\!\_\,}

\def\overbar#1#2#3{{%
	\setbox0=\hbox{\displaystyle{#1}}%
	\dimen0=\wd0
	\advance\dimen0 by -#2 
	\vbox {\nointerlineskip \moveright #3 \vbox{\hrule height 0.3pt width \dimen0}%
		\nointerlineskip \vskip 1.5pt \box0}%
}}

\newcommand{\shHom}{\mathscr{H}\hspace{-2.7pt}\mathit{om}}
\newcommand{\shExt}{\mathscr{E}\hspace{-1.5pt}\mathit{xt}}
\renewcommand{\Hom}{\operatorname{Hom}}

\newcommand{\rToDots}{\rTo[l>=3.3em]~{\raisebox{0.1pt}{$\,\dotsb\,$}}}
\newarrow{Equal}   =====
\newarrow{IntoBold}   {boldhook}---> 

\newcommand{\PP}{\mathbb{P}}

\newcommand{\famX}{\mathfrak{Y}}
\newcommand{\famZ}{\mathfrak{Z}}

\DeclareMathOperator{\codim}{codim}

\newcommand{\PProj}{\operatorname{Proj}}
\newcommand{\SSpec}{\operatorname{Spec}}
\newcommand{\Sym}{\operatorname{Sym}}
\newcommand{\Gr}{\operatorname{Gr}}

\newcommand{\OmX}[1]{\Omega_X^{#1}}

\newcommand{\OP}{\shf{O}_P}
\renewcommand{\OB}{\shf{O}_B}

\newcommand{\shH}{\shf{H}}
\newcommand{\shC}{\shf{C}}

\DeclareMathOperator{\DR}{DR}
\DeclareMathOperator{\Spenc}{\widetilde{\mathrm{DR}}}

\newcommand{\van}{_\mathit{van}}

\newcommand{\Mvan}{\Mmod \van}
\newcommand{\HMvan}{M \van}

\newcommand{\jl}{j_{\ast}}
\newcommand{\ql}{q_{\ast}}
\newcommand{\fl}{f_{\ast}}
\newcommand{\ju}{j^{\ast}}

\newcommand{\qu}{q^{\ast}}

\newcommand{\piu}{\pi^{\ast}}

\newcommand{\psiu}{\psi^{\ast}}

\newcommand{\pil}{\pi_{\ast}}

\newcommand{\QfamXH}{\QQ_{\famX}^H}

\newcommand{\decal}[1]{\lbrack #1 \rbrack}

\newcommand{\dfamX}{d_{\famX}}

\newcommand{\DDX}{\mathbb{D}_X}

\newcommand{\Dmod}{\mathcal{D}}
\newcommand{\Mmod}{\mathcal{M}}

\newcommand{\Nmod}{\mathcal{N}}

\makeatletter
\let\@@seccntformat\@seccntformat
\renewcommand*{\@seccntformat}[1]{%
  \expandafter\ifx\csname @seccntformat@#1\endcsname\relax
    \expandafter\@@seccntformat
  \else
    \expandafter
      \csname @seccntformat@#1\expandafter\endcsname
  \fi
    {#1}%
}
\newcommand*{\@seccntformat@subsection}[1]{%
  \textbf{\csname the#1\endcsname.}
}
\makeatother

\counterwithin{equation}{subsection}
\counterwithout{subsection}{section}
\counterwithin{figure}{subsection}

\newcommand{\subsecref}[1]{\S\ref{#1}}
\newcommand{\subsecsref}[2]{\S\ref{#1}--\ref{#2}}

\newtheorem{theorem}[equation]{Theorem}
\newtheorem*{theorem*}{Theorem}
\newtheorem{lemma}[equation]{Lemma}

\newtheorem{proposition}[equation]{Proposition}
\theoremstyle{definition}

\theoremstyle{remark}

\theoremstyle{plain}

\newcommand{\shE}{\shf{E}}
\newcommand{\shO}{\shf{O}}
\newcommand{\shF}{\shf{F}}
\newcommand{\shB}{\shf{B}}

\newcommand{\derR}{\mathbb{R}}
\newcommand{\derL}{\mathbb{L}}
\newcommand{\Gammast}{\Gamma_{\ast}}

\newcommand{\tensorL}{\underline{\tensor}}

\makeatletter
\let\old@caption\caption
\renewcommand*{\caption}[1]{%
	\setcounter{figure}{\value{equation}}%
	\stepcounter{equation}%
	\old@caption{#1}\relax%
}

\makeatother

\newcommand{\shEdual}{\shE^{\vee}}
\newcommand{\shGd}{\setbox0=\hbox{$\shG$} \shG\hspace{-0.9\wd0}\widehat{\phantom{G}}}
\newcommand{\shGrd}{\widehat{\shG^{\mathrm{r}}}}
\newcommand{\shGr}{\shG^{\mathrm{r}}}

\newcommand{\shSr}{\shS^{\mathrm{r}}}

\newcommand{\shG}{\shf{G}}
\newcommand{\shS}{\shf{S}}
\newcommand{\shOE}{\shf{O}_E}

\newcommand{\omegaS}{\omega_{\shS}}
\newcommand{\omegaX}{\omega_X}

\renewcommand{\argbl}{-}

\newcommand{\LshO}{L_{\shO}}
\newcommand{\LQQ}{L_{\QQ}}

\newcommand{\Db}{\operatorname{D}^{\mathrm{b}}}
\newcommand{\Dbqc}{\operatorname{D}_{\mathrm{qc}}^{\mathrm{b}}}
\newcommand{\Dbqcgr}{\operatorname{D}_{\mathrm{qc,gr}}^{\mathrm{b}}}
\newcommand{\Dbgr}{\operatorname{D}_{\mathrm{gr}}^{\mathrm{b}}}
\newcommand{\op}{^{\mathrm{op}}}

\begin{document}

\title{Local duality and mixed Hodge modules}
\author[C.~Schnell]{Christian Schnell}
\address{Department of Mathematics, Statistics \& Computer Science \\
University of Illinois at Chicago \\
851 South Morgan Street \\
Chicago, IL 60607}
\email{cschnell@math.uic.edu}
\subjclass[2000]{32C38, 14D07, 14B15}
\keywords{Mixed Hodge module, Local duality, Cohen-Macaulay property, Holonomicity,
Characteristic variety, De Rham complex}
\begin{abstract}
We establish a relationship between the graded quotients of a filtered
holonomic $\Dmod$-module, their sheaf-theoretic duals, and the characteristic
variety, in case the filtered $\Dmod$-module underlies a polarized Hodge module on a
smooth algebraic variety. The proof is based on Saito's result that the associated
graded module is Cohen-Macaulay, and on local duality on the cotangent bundle.
\end{abstract}
\maketitle


\section{Introduction}

In this paper, we prove a kind of duality theorem for filtered $\Dmod$-modules which
underlie polarized Hodge modules. Let $(\Mmod, F)$ be such a filtered $\Dmod$-module;
we show that there is a close relationship between, (1) the graded quotients
$\Gr_k^F \! \Mmod$, (2) their sheaf-theoretic duals, and (3) the characteristic variety
of $\Mmod$.

\subsection{Background}
\label{subsec:background}

This general result has its origins in the author's Ph.D.~dissertation
\cite{Schnell}, where the vanishing cohomology in the family of hyperplane sections
of a smooth projective variety was studied via residues. Very briefly, the situation
considered there is the following: Say $Y$ is a smooth projective variety of
dimension $n+1$; let $B$ be the space of hyperplane sections of $Y$, for an embedding
into projective space of sufficiently high degree. (The choice of symbols here is
awkward, but is to avoid confusion with the notation used in the remainder of the paper.) 
The universal family $\pi \colon \famX \to B$ has a smooth total space of dimension $\dfamX = d +
n$, where $d = \dim B$. On the open subset of $B$ corresponding to smooth hyperplane
sections, the vanishing cohomology of the fibers defines a polarized variation of
Hodge structure of weight $n$. Using M.~Saito's theory \cites{SaitoHM,SaitoMHM}, it
extends in a natural way to a polarized Hodge module $\HMvan$ on all of $B$; in fact,
as proved in \cite{BFNP}, $\HMvan$ is a direct summand
in the decomposition of $H^0 \pil \QfamXH \decal{\dfamX}$.

Now let $\bigl( \Mvan, F \bigr)$ be the filtered left $\Dmod$-module underlying
$\HMvan$. As usual, the finitely generated graded $\Sym \Theta_B$-module
\begin{equation} \label{eq:GrMvan}
	\Gr^F \! \Mvan = \bigoplus_{k \in \ZZ} \Gr_k^F \! \Mvan
\end{equation}
defines a coherent sheaf on the cotangent bundle of $B$, whose support is the
characteristic variety of $\Mvan$. In this geometric setting, it can be shown that
the projectivization of the characteristic variety is isomorphic to a subvariety
$\famZ \subseteq \famX$, whose points are the singular points in the fibers of
$\pi$. The fact that $\famZ$ is itself smooth then leads to the following
relationship between the sheaves $\shG_k = \Gr_k^F \! \Mvan$, their duals, and the characteristic
variety.

\begin{theorem*}
Let $\shC \van$ be the coherent sheaf on the projectivized cotangent bundle $P = \PP
\bigl( \Theta_B \bigr)$ associated to \eqref{eq:GrMvan}. Then for every $k \in
\ZZ$, there is an exact sequence
\begin{diagram}
	\shHom \bigl( \shG_{-n-k}, \OB \bigr) &\rIntoBold&
	\shG_k &\rTo& \pil \bigl( \shC \van \tensor \OP(k) \bigr)
	&\rOnto& \shExt^1 \bigl( \shG_{-n-k}, \OB \bigr);
\end{diagram}
up to a sign, the first map is given by $(2 \pi i)^{-n}$ times the intersection
pairing on the smooth fibers of $\pi$. For $i \geq 2$, we similarly have isomorphisms
\[
	R^{i-1} \pil \bigl( \shC \van \tensor \OP(k) \bigr) \simeq 
		\shExt^i \bigl( \shG_{-n-k}, \OB \bigr),
\]
again valid for every $k \in \ZZ$.
\end{theorem*}

More precisely, one has $\shC \van \simeq \psiu \omega_Y \tensor
\OP(n+1)$, where $\psi \colon \famZ \to B$ is the projection map. The
proof of the theorem in \cite{Schnell} exploited the fact that $\famZ$ is a local complete
intersection inside the product $B \times Y$, and proceeded through a careful analysis
of the Leray spectral sequence for the associated Koszul complex (in the spirit of
M.~Green).

\subsection{Summary of the paper}
The purpose of the present paper is to generalize the above result from $\HMvan$ to arbitrary
polarized Hodge modules; and, at the same time, to provide a more conceptual proof.  We will
show (in Theorem~\ref{thm:HM} below) that the theorem remains true for an arbitary
polarized Hodge module of weight $w$ on a smooth algebraic variety $X$, where now $n
= w - \dim X$. In fact, we improve the original statement by obtaining an exact
sequence (resp.\@ isomorphism) of graded $\shS$-modules, whose component in degree $k$
is the exact sequence (resp.\@ isomorphism) above; here, and in the following, $\shS = \Sym \Theta_X$.

Two main tools are used in the proof: a result by Saito and Kashiwara that $\Gr^F \!
\Mmod$ is Cohen-Macaulay as an $\shS$-module when $(\Mmod, F)$ underlies a polarized
Hodge module (see \cite{Saito-on}*{p.~55} for more information); and local duality on the
cotangent bundle $T_X^{\ast}$, relative to the zero section.

After recalling the relevant facts about local cohomology and local duality in
\subsecsref{subsec:loc-coh-vb}{subsec:loc-dual-vb}, we deduce our main result from
the Cohen-Macaulay property of $\Gr^F \! \Mmod$ in
\subsecsref{subsec:CM}{subsec:duality}. We also show, in
\subsecref{subsec:polarization}, that the initial map in the exact sequence is
induced by the polarization of $M$. Concretely, this means that on the dense open set
where $M$ is a polarized variation of Hodge structure of weight $n$, the map \[
	\shHom_{\OX} \bigl( \shG_{-n-k}, \OX \bigr) \to \shG_k
\]
is given by $(-1)^{d(d-1)/2} S(\argbl, \argbl)$, for $S$ the polarization and $d =
\dim X$. 

As one might expect, the exact sequence and the isomorphisms in Theorem~\ref{thm:HM}
are really part of an exact triangle in the derived category (of quasi-coherent
graded $\shS$-modules). This circumstance is useful when applying other functors, and
so we deduce it from the preceding sections in \subsecref{subsec:derived}.

\subsection{Applications}

In \subsecref{subsec:deRham}, we give a small application of Theorem~\ref{thm:HM} to
the study of the de Rham complex $\DR_X(\Mmod)$. This complex is naturally filtered
by subcomplexes $F_k \DR_X(\Mmod)$, and we show that the inclusion $F_{m-n-1}
\DR_X(\Mmod) \subseteq \DR_X(\Mmod)$ is a filtered quasi-isomorphism, where $m \in
\ZZ$ is such that $F_{-m} \Mmod = 0$. For example, when $M$ is the intermediate extension of
a polarized variation of Hodge structure of weight $n$, and $F^{n+1} \Mmod = F_{-n-1}
\Mmod = 0$, then $F_0 \DR_X(\Mmod) \subseteq \DR_X(\Mmod)$ is a quasi-isomorphism.
This fact plays a role in \cite{Schnell}, where properties of $(\Mvan, F)$ were used
to study normal functions associated to primitive Hodge classes.

In cases where enough information is available about the characteristic variety,
Theorem~\ref{thm:HM} can be used to get information about the sheaves $\shG_k$ and
$F_k \Mmod$. For example, it was shown in \cite{Schnell} that the sheaves $F_k
\Mvan$ in the range $-n \leq k \leq 0$ satisfy Serre's condition $S_p$ for large
values of $p$, provided that the degree of the embedding $Y \subseteq \PPn{d}$ is
sufficiently high. In particular, they are reflexive sheaves. This was done by
showing that the subset of $B$ corresponding to hypersurfaces with ``many''
singularities has large codimension in $B$; and then appealing to the theorem to
obtain the inequality $\codim \Supp \shExt^i \bigl( \shG_k, \OB \bigr) \geq i + p$ for $i > 0$, which
is equivalent to Serre's condition.

\section{Local duality on vector bundles}

Let $X$ be a smooth algebraic variety (or any quasi-compact variety where every coherent sheaf is
the quotient of a locally free one), and $E \to X$ a vector bundle of rank $d \geq
1$. We review several facts about local cohomology on $E$ with support in the zero
section, as well as about local duality. The case when $E$ is an affine space is
well-known, and is explained very clearly in Appendix~1 of D.~Eisenbud's book
\cite{Eisenbud}*{pp.~187--199}. Short proofs are included here for the sake of
completeness; they are mostly straightforward generalizations of the ones in
\cite{Eisenbud}.

\subsection{Local cohomology on a vector bundle}
\label{subsec:loc-coh-vb}

Let $\shE$ be a locally free sheaf on $X$ of rank $d \geq 1$. The symmetric algebra
\[
	\shS = \Sym \shEdual = \bigoplus_{k \geq 0} \Sym^k \shEdual
\]
is a sheaf of graded $\OX$-algebras, and $E = \SSpec \shS$ is the vector
bundle corresponding to $\shE$. The map $f \colon E \to
X$ is affine, and we have $\shS \simeq \fl \shOE$. Quasi-coherent sheaves on $E$ are in
one-to-one correspondence with quasi-coherent $\shS$-modules on $X$; given a sheaf of
$\shS$-modules $\shG$, we let $\shG_E$ be the corresponding sheaf on $E$, so that $\fl
\shG_E \simeq \shG$.

The original variety $X$ is naturally embedded into $E$ by the zero section of the
vector bundle.  Let $\shF$ be any quasi-coherent sheaf on $E$. The subsheaf
$\shH_X^0(\shF)$ consists of all sections of $\shF$ whose support is contained in the
zero section $X \subseteq E$. Then $\shH_X^0$ is a left-exact
functor on quasi-coherent $\shO_E$-modules, and its $i$-th right-derived functor is
denoted by the symbol $\shH_X^i$; we call the sheaf $\shH_X^i(\shF)$ the
$i$-th \define{local cohomology sheaf} of $\shF$ with support in the zero section of
$E$. The corresponding quasi-coherent $\shS$-module is $\fl \shH_X^i(\shF)$; when
$\shG$ is a graded $\shS$-module, the local cohomology modules $\fl \shH_X^i \bigl(
\shG_E \bigr)$ are naturally graded $\shS$-modules as well. More information about local
cohomology sheaves can be found in \cite{SGA2}*{Expos\'es~I and II on pp.~5--26}, in
the expected greater generality.

We also consider the projectivization of the vector bundle, given by $P = \PProj
\shS$, together with the projection map $\pi \colon P \to X$ (see
\cite{EGA2}*{Chapitre~II, \S3}
for details). As usual, we write $\OP(1)$ for the universal line bundle on $P$. A
finitely generated graded $\shS$-module
\[
	\shG = \bigoplus_{k \in \ZZ} \shG_k,
\]
defines a coherent sheaf $\shG_P$ on the projective bundle $P$. We let $\shG(m)$
be the graded $\shS$-module with $\shG(m)_k = \shG_{m+k}$; evidently, $\OP(1)$ is the coherent
sheaf associated to $\shS(1)$. For $\shF$ a coherent sheaf on $P$, and $i \geq 0$, we
have a graded $\shS$-module
\[
	R^i \Gammast(\shF) = \bigoplus_{k \in \ZZ} 
		R^i \pil \bigl( \shF \tensor \OP(k) \bigr).
\]
Since $\OP(1)$ is relatively ample, the natural map $\shG \to \Gammast(\shG_P)$
is an isomorphism in large degrees; its behavior in arbitary degrees is related to
the local cohomology sheaves of $\shG_E$, as shown by the following proposition.

\begin{proposition} \label{prop:local-cohomology}
Let $\shG$ be a finitely generated graded $\shS$-module on $X$. Then there is an
exact sequence
\begin{diagram}[l>=2em]
	\fl \shH_X^0 \bigl( \shG_E \bigr) &\rIntoBold& \shG &\rTo& 
		\Gammast(\shG_P) &\rOnto& \fl \shH_X^1 \bigl( \shG_E \bigr)
\end{diagram}
of graded $\shS$-modules. Moreover, for each $i \geq 2$, we have an isomorphism
\[
	\fl \shH_X^i \bigl( \shG_E \bigr) \simeq R^i \Gammast(\shG_P)
\]
again of graded $\shS$-modules.
\end{proposition}

\begin{proof}
For the convenience of the reader, we briefly review the argument. Consider the following
commutative diagram of maps.
\begin{diagram}[tight,width=2.5em]
E - X &\rTo^j& E \\
\dTo^q && \dTo_f \\
P &\rTo^{\pi}& X
\end{diagram}
By \cite{SGA2}*{Corollaire~2.11 on p.~16}, we have an exact sequence
\begin{equation} \label{eq:seq-i}
\begin{diagram}
0 &\rTo& \shH_X^0 \bigl( \shG_E \bigr) &\rTo& \shG_E &\rTo& \jl \ju \shG_E &\rTo&
	\shH_X^1 \bigl( \shG_E \bigr) &\rTo& 0.
\end{diagram}
\end{equation}
Now $\shG$ is graded, and so we have $\ju \shG_E \simeq \qu \shG_P$. Using the projection
formula, we then find that
\[
	\fl \jl \ju \shG_E \simeq \pil \left( \shG_P \tensor \ql \shO_{E - X} \right)
		\simeq \pil \left( \shG_P \tensor \bigoplus_{k \in \ZZ} \OP(k) \right)
		\simeq \Gammast(\shG_P).
\]
Applying the exact functor $\fl$ to the sequence in \eqref{eq:seq-i}, and noting that
$\fl \shG_E \simeq \shG$, we obtain the first half of the
proposition. The second half follows by similar considerations from the isomorphism
$\shH_X^i \bigl( \shG_E \bigr) \simeq R^{i-1} \jl \ju \shG_E$ 
for $i \geq 2$, also given in \cite{SGA2}*{p.~16}.
\end{proof}

\subsection{Local duality on a vector bundle}
\label{subsec:loc-dual-vb}

Given a graded $\shS$-module $\shG$, we define its \define{graded dual} to be
\begin{equation} \label{eq:graded-dual}
	D(\shG) = \bigoplus_{k \in \ZZ} \shHom_{\OX} \bigl( \shG_{-k}, \OX \bigr).
\end{equation}
This is again a graded $\shS$-module, with the summand $\shHom_{\OX} \bigl(
\shG_{-k}, \OX \bigr)$ in degree $k$; the action of $\shS$ is given by the rule $(s
\cdot \phi)(g) = \phi(s g)$. The $i$-th derived functor of $D$ is then evidently
\[
	D^i(\shG) = \bigoplus_{k \in \ZZ} \shExt_{\OX}^i \bigl( \shG_{-k}, \OX \bigr).
\]
Note that even when $\shG$ is finitely generated as an $\shS$-module, $D^i(\shG)$ is usually
not; unless, of course, $\shG$ actually has finite length.

In analogy with the canonical line bundle of projective space, we also introduce the graded
$\shS$-module $\omegaS =  \det \shEdual \tensor_{\OX} \shS(-d)$, whose graded piece in
degree $k$ is $\det \shEdual \tensor \Sym^{k-d} \shEdual$; here $d$ is
still the rank of the vector bundle. Then $\omega_{P/X}$ is the sheaf
associated to $\omegaS$, because a simple calculation with the Euler sequence for
$\pi \colon P \to X$ shows that
\begin{equation} \label{eq:omega-PX}
	\omega_{P / X} \simeq \bigl( \piu \det \shEdual \bigr) \tensor \OP(-d).
\end{equation}

The second important result about local cohomology sheaves on a vector bundle is the
following duality theorem, known as \define{graded local duality} in the case of an
affine space.

\begin{proposition} \label{prop:local-duality}
Let $\shG$ be a finitely generated graded $\shS$-module on $X$.
Then there is a convergent fourth-quadrant spectral sequence of graded $\shS$-modules,
\[
	E_2^{p,q} = D^p \Bigl( \shExt_{\shS}^{-q} \bigl( \shG, \omegaS \bigr) \Bigr)
		\Longrightarrow \shH_X^{d+p+q} \bigl( \shG_E \bigr),
\]
functorial in the sheaf $\shG$.
\end{proposition}

The following notion will be useful during the proof. A graded $\shS$-module will be
called \define{basic} if it is a finite direct sum of modules of the form $\shB
\tensor_{\OX} \shS(m)$, with $\shB$ a locally free $\OX$-module. The local cohomology
sheaves are easy to describe in that case.

\begin{lemma} \label{lem:res-basic}
Let $\shF = \shB \tensor \shS(m)$ be a basic graded $\shS$-module. Then
\begin{equation} \label{eq:shH-basic}
	\fl \shH_X^i \bigl( \shF_E \bigr) \simeq
		\begin{cases}
			D \Bigl( \shHom_{\shS} \bigl( \shF, \omegaS \bigr) \Bigr) 
					&\text{for $i = d$,} \\
			0 &\text{otherwise.}
		\end{cases}	
\end{equation}
\end{lemma}
\begin{proof}
Since $\shB$ is locally free, one is quickly reduced to the case $\shF = \shS$, where
$\shF_E = \shOE$. When $d = 1$, the assertion follows immediately from the exact sequence in
Proposition~\ref{prop:local-cohomology}. Thus we may assume from now on that $d \geq
2$. Since $\pil \OP(k) \simeq \shS_k$, while $R^i \pil \OP(k) = 0$ for $1 \leq i \leq
d - 1$, Proposition~\ref{prop:local-cohomology} shows that
$\shH_X^i \bigl( \shOE \bigr) = 0$ for $i \neq d$. For $i = d$, we let
\begin{equation} \label{eq:shH}
	\shH = \fl \shH_X^d \bigl( \shOE \bigr) \simeq 
		\bigoplus_{k \in \ZZ} R^{d-1} \pil \OP(k).
\end{equation}
At this point, one can easily obtain the isomorphism in \eqref{eq:shH-basic} by using
duality for the morphism $\pi \colon P \to X$. Following \cite{Eisenbud}*{p.~191}, we
shall give a more concrete derivation using a \v{C}ech complex, because this serves
to show the $\shS$-module structure on $\shH$ more clearly.

Let $U \simeq \Spec A$ be an affine open subset of $X$ over which the vector bundle $E$ is trivial. Then
$S = \Gamma(U, \shS) \simeq A \lbrack t_1, \dotsc, t_d \rbrack$ as graded
$A$-algebras, and $f^{-1}(U) \simeq \Spec S$. Let $I = A t_1 + \dotsb + A t_d$ be the
irrelevant ideal, and $V(I) = \Spec(S/I)$. As in the proof of
Proposition~\ref{prop:local-cohomology}, the local cohomology module $H = \Gamma(U, \shH)$ that we need is
\[
	H = H_{V(I)}^d \bigl( \Spec S, \shO \bigr) \simeq 
		H^{d-1} \bigl( \Spec S - V(I), \shO \bigr);
\]
it can be computed by the \v{C}ech complex for the standard open cover of $\Spec S -
V(I)$. As a graded $S$-module, $H$ is therefore isomorphic to the cokernel of the map
\[
	\bigoplus_{i=1}^d S 
		\bigl\lbrack (t_1 \dotsm t_{i-1} t_{i+1} \dotsm t_d)^{-1} \bigr\rbrack
		\to S \bigl\lbrack (t_1 \dotsm t_d)^{-1} \bigr\rbrack.
\]
Thus $H_k$ is generated by elements of the form $F/(t_1 \dotsm t_d)^m$, where $F$ is
a homogeneous polynomial of degree $k + dm$. From degree
considerations, we see that $H_k = 0$ for $k > -d$, while $H_{-d} \simeq A$,
generated by $(t_1 \dotsm t_d)^{-1}$. We now have a map
\[
	H_k \to \Hom_A \bigl( S_{-k-d}, H_{-d} \bigr),
\]
by sending an element $F/(t_1 \dotsm t_d)^m \in H_k$ to the functional $G \mapsto
FG/(t_1 \dotsm t_d)^m$, for $G \in S_{-k-d}$. This is easily seen to be an
isomorphism; moreover, since the $S$-module structure on $H$ is simply given by multiplication, we
obtain
\[
	H \simeq \bigoplus_{k \in \ZZ} \Hom_A \bigl( S_{-k-d}, H_{-d} \bigr) 
		= H_{-d} \tensor_A D \bigl( S(-d) \bigr)
\]
as graded $S$-modules. As written, the isomorphism is coordinate-independent, and so
we get a global isomorphism of graded $\shS$-modules
\[
	\shH \simeq \shH_{-d} \tensor_{\OX} D \bigl( \shS(-d) \bigr).
\]
Note that $\shH_{-d}$ has rank one, and is therefore a line bundle on $X$. Because of
\eqref{eq:shH}, we find that $\shH_{-d} \simeq R^{d-1} \pil \OP(-d) \simeq \det
\shE$, and this concludes the proof.
\end{proof}

\begin{lemma} \label{lem:resolution}
Let $\shG$ be a finitely generated graded $\shS$-module on $X$. Then $\shG$ can be
resolved in the form
\begin{diagram}
	\dotsb &\rTo& \shF^{-2} &\rTo& \shF^{-1} &\rTo& \shF^0 &\rOnto& \shG
\end{diagram}
by basic graded $\shS$-modules $\shF^i$.
\end{lemma}
\begin{proof}
It suffices to show that every finitely generated graded $\shS$-module $\shG$ admits a
surjection by a basic one. Since $\shG$ is finitely generated, and
$X$ is quasi-compact, there is a finite set $F \subseteq \ZZ$ such that 
\[
	\bigoplus_{k \in F} \shG_k \tensor_{\OX} \shS(-k) \to \shG
\]
is surjective. Each $\shG_k$ is a coherent sheaf of $\OX$-modules, and because $X$ is
smooth, there is a locally free sheaf $\shB_k$ mapping onto
$\shG_k$. Then $\shF^0 = \bigoplus_{k \in F} \shB_k \tensor \shS(-k)$ is a basic
module mapping onto $\shG$.
\end{proof}

Here is the proof of Proposition~\ref{prop:local-duality}.
\begin{proof}
Let $\shG$ be any finitely generated graded $\shS$-module. According to Lemma~\ref{lem:resolution},
there is a complex $\shF^{\bullet}$ of basic graded $\shS$-modules resolving $\shG$. The
local cohomology sheaves of $\shG$ are therefore computed by a spectral sequence
\[
	E_1^{p,q} = \fl \shH_X^q \bigl( \shF^p \bigr) \Longrightarrow
		\shH_X^{p+q} \bigl( \shG_E \bigr).
\]
From \eqref{eq:shH-basic}, all but one row of the $E_1$-page is zero, and so
$\shH_X^{d+i} \bigl( \shG_E \bigr)$ is isomorphic to the cohomology in degree $i$ of the complex
\[
	\fl \shH_X^d \bigl( \shF^{\bullet} \bigr) \simeq 
		D \bigl( \shHom_{\shS} \bigl( \shF^{\bullet}, \omegaS \bigr) \Bigr).
\]
The spectral sequence is now simply the one for the composition of the two contravariant
functors $\shHom_{\shS} \bigl( \argbl, \omegaS \bigr)$ and $D$.
\end{proof}

\subsection{Local duality for Cohen-Macaulay modules}
\label{subsec:loc-dual-CM}

Now suppose that the graded $\shS$-module $\shG$ is in addition \define{Cohen-Macaulay} of
dimension $d$; that is to say, the associated coherent sheaf $\shG_E$ is
Cohen-Macaulay on $E$, with purely $d$-dimensional support. Consequently,
$\shExt_{\shS}^q(\shG, \shS) = 0$ unless $q = d$; let 
\[
	\shGd = \shExt_{\shS}^d \bigl( \shG, \omegaS \bigr) = 
		\shExt_{\shS}^d \bigl( \shG, \shS(-d) \bigr) \tensor_{\OX} \det \shEdual
\]
be the dual $\shS$-module. The spectral sequence in
Proposition~\ref{prop:local-duality} degenerates at the $E_2$-page, because it has
only one nonzero row, and we find that
\[
	\fl \shH_X^p \bigl( \shG_E \bigr) 
		\simeq D^p \bigl( \shGd \bigr) =
		\bigoplus_{k \in \ZZ} \shExt_{\OX}^p \bigl( \shGd_{-k}, \OX \bigr) 
\]
for all $p \geq 0$. In combination with Proposition~\ref{prop:local-cohomology}, we
now get the following result.

\begin{theorem} \label{thm:CM-duality}
Let $\shG$ be a finitely generated graded $\shS$-module, which is Cohen-Macaulay of
dimension $d = \rk \shE$. Let $\shGd = \shExt_{\shS}^d(\shG, \omegaS)$ be the dual
module. Then there is an exact sequence
\begin{diagram}[l>=2em]
	D \bigl( \shGd \bigr) &\rIntoBold& \shG &\rTo& 
	\Gammast(\shG_P) &\rOnto& D^1 \bigl( \shGd \bigr)
\end{diagram}
of graded $\shS$-modules. Moreover, for each $i \geq 2$, we have an isomorphism
\[
	D^i \bigl( \shGd \bigr) \simeq R^i \Gammast(\shG_P),
\]
again respecting the graded $\shS$-module structure on both sides.
\end{theorem}

\section{Polarized Hodge modules}

From now on, let $X$ be a nonsingular complex algebraic variety of dimension $d \geq
1$. The cotangent bundle $E = T_X^{\ast}$ is then a vector bundle of rank $d$ on
$X$; as in \subsecref{subsec:loc-coh-vb}, we let $P = \PP(\Theta_X)$ be its
projectivization, $\pi \colon P \to X$ the natural map, and $\OP(1)$ the universal
line bundle on $P$.  Also let $\shS = \Sym \Theta_X$. Note that $\det \shE =
\omegaX$, and so $\omegaS = \omegaX^{-1} \tensor_{\OX} \shS(-d)$.

\subsection{The Cohen-Macaulay property}
\label{subsec:CM}

Let $M$ be a polarized Hodge module on $X$ of weight $w$. We write $(\Mmod, F)$ for the
underlying filtered holonomic \emph{left} $\Dmod$-module, and consider the graded $\shS$-module
\[
	\shG = \Gr^F \! \Mmod = \bigoplus_{k \in \ZZ} \Gr_k^F \! \Mmod.
\]
As before, $\shG_E$ is the corresponding coherent sheaf on $E$, and $\shG_P =
\shC(\Mmod, F)$ what might be called the ``characteristic sheaf'' of the
$\Dmod$-module, defined on $P$.  The support of the sheaf $\shG_E$ is the characteristic
variety of the $\Dmod$-module \cite{Borel}*{p.~212-3}; it is a cone in $E$, and the
support of $\shG_P$ is the projectivization of that cone.
Since $\Mmod$ is holonomic, its characteristic variety is of pure dimension $d$. But
because $M$ is a Hodge module, much more is true: in fact, Saito has shown that
$\shG$ is always a Cohen-Macaulay module over $\shS$ (in \cite{SaitoHM}*{Lemme~5.1.13 on
p.~958}). Consequently, the sheaf $\shG_E$ is Cohen-Macaulay of dimension $d$ on $E$,
and so Theorem~\ref{thm:CM-duality} may be applied to it.

Moreover, the dual $\shGd = \shExt_{\shS}^d \bigl( \shG, \omegaS \bigr) = 
\shExt_{\shS}^d \bigl( \shG, \shS(-d) \bigr) \tensor \omegaX^{-1}$ can be computed
explicitely in this case, since $M$ is polarized. To do this, we need the following
bit of notation: For a graded $\shS$-module $\shG$, we define $\shGr$ to be the same
graded $\OX$-module as $\shG$, but with the action of $\shS$ changed so that sections
of $\shS_k$ act with an extra factor of $(-1)^k$. (This corresponds to pulling
$\shG_E$ back by the involution $e \mapsto -e$ of the cotangent bundle.)
Evidently, we have $\shGr \simeq \shG \tensor_{\shS} \shSr$.

Now let $M' = \DDX(M)$ be the Verdier dual of the Hodge module, and $(\Mmod', F) =
\DDX(\Mmod, F)$ the underlying filtered left $\Dmod$-module. According to
\cite{Saito-on}*{p.~54--5}, we have
\[
	\Mmod' = \shExt_{\Dmod_X}^d \bigl( \Mmod, \Dmod_X \tensor_{\OX} \omegaX^{-1} \bigr),
\]
where the filtration on $\Dmod_X \tensor \omegaX^{-1}$ is given by $F_p \bigl(
\Dmod_X \tensor \omegaX^{-1} \bigr) = F_{p-2d} \Dmod_X \tensor \omegaX^{-1}$. Because
of strictness (which is equivalent to the Cohen-Macaulay property), we can pass to
the associated graded modules to obtain
\[
	\Gr^F \! \Mmod' = \shExt_{\shS}^d \bigl( \shGr, \shS(-2d) \bigr) \tensor \omegaX^{-1}
		\simeq \shGrd(-d).
\]
The change in module structure from $\shG$ to $\shGr$ happens because, in computing
$\Mmod'$, one is really passing from $\Mmod$ to
the associated right $\Dmod$-module $\omegaX \tensor_{\OX} \Mmod$, and the right action of
$\Theta_X$ on $\omegaX \tensor \Mmod$ is given by the rule
\[
	(\omega \tensor m) \cdot \xi = (\omega \cdot \xi) \tensor m 
		- \omega \tensor (\xi m),
\]
thus introducing an additional sign when passing to the graded module.

A polarization on $M$ is an isomorphism $M \simeq \DDX(M)(-w)$, where $w$ is the
weight of $M$. If $M$ is polarized, we thus have
\[
	(\Mmod, F) \simeq \DDX(\Mmod, F)(-w) = \bigl( \Mmod', F_{\bullet + w} \bigr).
\]	
When combined with the isomorphism above, this gives
\begin{equation} \label{eq:Saito-pol}
	\shG \simeq \bigl( \Gr^F \! \Mmod' \bigr)(w) \simeq \shGrd(w-d),
\end{equation}
or in other words, $\shGd \simeq \shGr(d-w)$.

\subsection{Duality for polarized Hodge modules}
\label{subsec:duality}

We now obtain from Theorem~\ref{thm:CM-duality} the following result about polarized Hodge modules.

\begin{theorem} \label{thm:HM}
Let $M$ be a polarized Hodge module of weight $w = d + n$ on the nonsingular $d$-dimensional
complex algebraic variety $X$. Let $(\Mmod, F)$ be the underlying filtered
left $\Dmod$-module, and write $\shG = \Gr^F \! \Mmod$ for the associated graded $\shS$-module.
Also let $\shC = \shC(\Mmod, F)$ be the corresponding coherent sheaf on $P =
\PP(\Theta_X)$. Then there is an exact sequence
\begin{equation} \label{eq:HM-seq}
\begin{diagram}[l>=2em]
	D \bigl( \shGr(-n) \bigr) &\rIntoBold& \shG &\rTo& 
	\Gammast(\shC) &\rOnto^{\hspace{1.8em}}& D^1 \bigl( \shGr(-n) \bigr)
\end{diagram}
\end{equation}
of graded $\shS$-modules on $X$. Similarly, for each $i \geq 2$, we have an isomorphism
\[
	D^i \bigl( \shGr(-n) \bigr) \simeq R^i \Gammast(\shC)
\]
of graded $\shS$-modules.
\end{theorem}

The graded $\shS$-module $D^i \bigl( \shGr(-n) \bigr)$ is easily described.
Indeed, for any integer $k$, its graded piece in degree $k$ is
\[
	D^i \bigl( \shGr(-n) \bigr)_k = 
		\shExt_{\OX}^i \bigl( \shG_{-n-k}, \OX \bigr) =
		\shExt_{\OX}^i \bigl( \Gr_{-n-k}^F \! \Mmod, \OX \bigr).
\]
We thus get, for each value of $k$, an exact sequence
\begin{equation} \label{eq:SESk}
\begin{diagram}[l>=1.8em]
	\shHom \bigl( \shG_{-n-k}, \OX \bigr) &\rIntoBold&
	\shG_k &\rTo& \pil \bigl( \shC \tensor \OP(k) \bigr)
	&\rOnto& \shExt^1 \bigl( \shG_{-n-k}, \OX \bigr).
\end{diagram}
\end{equation}

\begin{note}
In the case of $\Mvan$, we have $w = \dfamX = n+d$, with $n$ the dimension of the
hyperplane sections, and so we recover precisely the result mentioned in
\subsecref{subsec:background}.
\end{note}

\subsection{The role of the polarization}
\label{subsec:polarization}

The derivation of Theorem~\ref{thm:HM} shows that the map $D \bigl( \shGr(-n) \bigr)
\to \shG$ in the exact sequence \eqref{eq:HM-seq} is induced by the polarization of
the Hodge module; in fact, the isomorphism $M \simeq \DDX(M)(-w)$ is exactly what was
used to pass from Proposition~\ref{prop:local-duality} to Theorem~\ref{thm:HM}. To
see this more clearly, we first consider the case when $M$ comes from a polarized
variation of Hodge structure.

So let $\bigl( \LshO, \nabla, F, \LQQ, S \bigr)$ be a polarized variation of Hodge
structure of weight $n$. The flat connection $\nabla$ makes $\LshO$ into a left
$\Dmod$-module, which we denote by $\Mmod$; it is filtered by setting $F_k \Mmod =
F^{-k} \LshO$, because of Griffiths transversality. In this special case, the graded
$\shS$-module $\shG = \Gr^F \! \Mmod$ is of finite length.

Now consider the polarization $S \colon \LQQ \tensor \LQQ \to \QQ(-n)$ of the variation.
By definition, we have $S \bigl( F^p \LshO, F^q \LshO \bigr) = 0$ for $p + q > n$;
thus $S$ descends to a non-degenerate bilinear pairing between $\Gr_F^{-k} \LshO$ and
$\Gr_F^{n+k} \LshO$ for all $k$. We get an isomorphism
\begin{equation} \label{eq:iso-VHS}
	\bigoplus_{k \in \ZZ} \Gr_k^F \! \Mmod \simeq 
		\bigoplus_{k \in \ZZ} \shHom_{\OX} \bigl( \Gr_{-n-k}^F \! \Mmod, \OX \bigr).
\end{equation}
Moreover, $S$ is flat for the connection $\nabla$, and so
\[
	d S(\lambda_1, \lambda_2) = S(\nabla \lambda_1, \lambda_2) 
		+ S(\lambda_1, \nabla \lambda_2)
\]
for all sections $\lambda_1, \lambda_2$ of $\LshO$. When $\lambda_1$ is a section
of $\Gr_{k-1}^F \! \Mmod = \Gr_F^{-k+1} \LshO$, and $\lambda_2$ a section of $\Gr_{-n-k}^F
\Mmod = \Gr_F^{n+k} \LshO$, we therefore have
\[
	0 = \xi \cdot S(\lambda_1, \lambda_2) = S(\xi \cdot \lambda_1, \lambda_2)
		+ S(\lambda_1, \xi \cdot \lambda_2)
\]
for arbitrary vector fields $\xi$. This shows that \eqref{eq:iso-VHS} is compatible
with the action by $\shS$, provided that sections of $\shS_k$ act on the
right-hand side with an extra factor of $(-1)^k$. In the notation used in
\subsecref{subsec:CM}, the
polarization therefore determines an isomorphism of graded $\shS$-modules $\shG \simeq
D \bigl( \shGr(-n) \bigr)$.

Let $M$ be the polarized Hodge module associated to the variation
\cite{SaitoHM}*{Theorem~5.4.3 on p.~989}; its weight is $w =
d + n$. As expected, the map in \eqref{eq:HM-seq} is the one given by the
polarization $S$, up to a sign factor.

\begin{lemma} \label{lem:polarization}
Let $M$ be the Hodge module associated to a polarized variation of Hodge structure
$\bigl( \LshO, \nabla, F, \LQQ, S \bigr)$ of weight $n$, with $S \colon \LQQ \tensor \LQQ \to
\QQ(-n)$ the polarization. Then the map $D \bigl( \shGr(-n) \bigr) \to \shG$ in
\eqref{eq:HM-seq} is an isomorphism, whose inverse
\[
	\bigoplus_{k \in \ZZ} \Gr_F^k \LshO \to
		\bigoplus_{k \in \ZZ} \shHom_{\OX} \bigl( \Gr_F^{n-k} \LshO, \OX \bigr),
\]
is given by the formula
\[
	\lambda \mapsto (-1)^{d(d-1)/2} S(\lambda, \argbl)
\]
for $\lambda \in \Gr_F^k \LshO$.
\end{lemma}

The proof of this lemma is given in \subsecref{subsec:pol-proof} below. 

Now we return to the case of a general polarized Hodge module $M$ of weight $w = d +
n$. There is always a dense Zariski-open subset $U \subseteq X$ on which $M$ is the
Hodge module associated to a polarized variation of Hodge structure of weight $n$
\cite{SaitoHM}*{Lemme~5.1.10 on p.~957}. Over $U$, therefore, the map $D \bigl(
\shGr(-n) \bigr) \to \shG$ is the one in Lemma~\ref{lem:polarization}. But then the
same has to be true on all of $X$. In other words, through the first map in
\eqref{eq:SESk}, a local section $\sigma$ of
$\shHom(\Gr_{-n-k}^F \! \Mmod, \OX)$ determines a local section $i(\sigma)$ of
$\Gr_k^F \! \Mmod$. Lemma~\ref{lem:polarization} shows that, at least at points of $U$, we
have
\[
	(-1)^{d(d-1)/2} S \bigl( i(\sigma), m \bigr) = \sigma(m)
\]
for every local section $m$ of $\Gr_{-n-k}^F \! \Mmod$. But since both sides are
holomorphic, and $U$ is dense in $X$, this identity has to hold at points of $X
\setminus U$ as well.

A different way to think about this is the following. Over $U$, any section of
$\Gr_k^F \! \Mmod$ determines a linear functional on $\Gr_{-n-k}^F \! \Mmod$. We can
thus think of the sheaf $\pil \bigl( \shC(\Mmod, F) \tensor \OP(k) \bigr)$ in
\eqref{eq:SESk}, whose support is contained in the complement of $U$, as giving the
obstructions for that functional to extend over points of $X \setminus U$.

\subsection{Proof of the lemma}
\label{subsec:pol-proof}

The first assertion in Lemma~\ref{lem:polarization} is very easy to prove. Indeed,
the characteristic variety of $\Mmod$ is the zero section, and $\shG$ has finite
length as an $\shS$-module. We thus have $\shC(\Mmod, F) = 0$, and so the map $D
\bigl( \shGd \bigr) \to \shG$ in Theorem~\ref{thm:CM-duality} is an isomorphism in
this case. It follows that the first map in \eqref{eq:HM-seq} is also an isomorphism. 
Now, given Saito's description of the polarization in \cite{SaitoHM}*{Lemme~5.4.2 on
p.~989}, it is certainly believable that the isomorphism should be given by $(-1)^{d(d-1)/2}
S$ as in \eqref{eq:iso-VHS}; however, it seemed advantageous to write down a more detailed
proof. This is the purpose of the present section; it involves looking more closely
at Saito's construction.

Since it becomes necessary to use both left and right $\Dmod$-modules
here, we shall introduce the following notation. As in \subsecref{subsec:polarization}, the
filtered left $\Dmod$-module determined by $\bigl( \LshO, \nabla, F \bigr)$ will be
denoted by $(\Mmod, F)$, and the associated graded $\shS$-module by
$\shG = \Gr^F \! \Mmod$. The corresponding right $\Dmod$-module is then $\Nmod =
\omega_X \tensor_{\OX} \Mmod$, with $\Dmod$-module structure defined by the rule
\begin{equation} \label{eq:left-right}
	(\omega \tensor m) \cdot \xi = 
		(\omega \cdot \xi) \tensor m - \omega \tensor (\xi \cdot m),
\end{equation}
for $\xi$ any section of $\Theta_X$. The filtration is given by $F_p \Nmod = \omega_X
\tensor F_{p+d} \Mmod$; together with \eqref{eq:left-right}, this shows that
\[
	\Gr^F \! \Nmod \simeq \omega_X \tensor_{\OX} \shGr(d),
\]
in the notation of \subsecref{subsec:CM}.

Saito proves that $(\Nmod, F)$ has a canonical resolution by induced $\Dmod$-modules
\cite{SaitoHM}*{Lemme~2.1.6 on p.~877}.
It is constructed by taking the Spencer complex $\Spenc(\Nmod, F)$ (isomorphic to the
de Rham complex $\DR(\Mmod, F)$ for the original left $\Dmod$-module), and tensoring
on the right by $(\Dmod_X, F)$; the augmentation map
\[
	\Spenc(\Nmod, F) \tensor_{\OX} (\Dmod_X, F) \to (\Nmod, F)
\]
is a filtered quasi-isomorphism. The associated graded complex
\[
	\Gr^F \Bigl( \Spenc(\Nmod, F) \tensor_{\OX} (\Dmod_X, F) \Bigr)
\]
then provides a canonical resolution of $\Gr^F \! \Nmod$ by basic graded
$\shS$-modules (as in Lemma~\ref{lem:res-basic}), because $\Gr^F \! \Nmod$ is locally
free over $\OX$. We let
\[
	\shF^{\bullet} = 
		\omegaX^{-1} \tensor_{\OX} 
			\Gr^F \Bigl( \Spenc(\Nmod, F) \tensor_{\OX} (\Dmod_X, F) \Bigr)(-d),
\]
which resolves $\shGr$ by basic graded $\shS$-modules.

Saito's construction of the isomorphism $(\Mmod, F) \simeq \DDX(\Mmod, F)(-w)$ is
the following. He shows that $(-1)^{d(d-1)/2} S$ gives a filtered
quasi-isomorphism
\[
	\Spenc(\Nmod, F) \to \shHom_{\OX}^F \Bigl( \Spenc(\Nmod, F), 
		(\omegaX, F) \decal{d} \Bigr)(-w);
\]
note that, in this case only, the filtration on $\omegaX$ is such that $\Gr_k^F \omegaX =
0$ for $k \neq 0$. Passing to induced modules, one gets a filtered quasi-isomorphism
\[
	\Spenc(\Nmod, F) \tensor (\Dmod_X, F) \to
		\shHom_{\Dmod_X} \Bigl( \Spenc(\Nmod, F) \tensor (\Dmod_X, F),
			(\omega_X \tensor \Dmod_X, F) \decal{d} \Bigr)(-w).
\]
Here $\omega_X \tensor_{\OX} \Dmod_X$ has two different structures as a right
$\Dmod$-module; one is used when applying $\shHom_{\Dmod_X}(\argbl, \argbl)$, and the other to make
the right-hand side into a complex of right $\Dmod$-modules. Since $\Gr^F \! \Nmod$
is locally free over $\OX$, that complex computes the Verdier
dual $\DDX(\Nmod, F)(-w)$; seeing that the left-hand side is quasi-isomorphic to $(\Nmod,
F)$, one has the desired polarization, on the level of filtered right
$\Dmod$-modules.

Using the strictness property of the right-hand side (again because $\shG$ is
Cohen-Macaulay), we can now pass to the associated graded complexes. Noting that a
Tate twist operates by $\bigl( \Mmod, F \bigr)(-w) = \bigl( \Mmod, F_{\bullet + w}
\bigr)$, we obtain a quasi-isomorphism
\[
	\omegaX \tensor \shF^{\bullet}(d) \to 
		\shHom_{\shS} \Bigl( \omegaX \tensor \shF^{\bullet}(d), 
			\omega_X \tensor \shS(w) \decal{d} \Bigr) \tensor_{\shS} \shSr;
\]
the change in module structure by $\shSr$ happens because of
the two different $\Dmod$-module structures on $\omegaX \tensor_{\OX} \Dmod_X$. After some
cancellation, and with the abbreviation $\omegaS = \omegaX^{-1}
\tensor \shS(-d)$, therefore,
\begin{equation} \label{eq:pol-quasi}
	\shF^{\bullet} \tensor_{\shS} \shSr \to 
		\shHom_{\shS} \Bigl( \shF^{\bullet}(-n), \omegaS \decal{d} \Bigr)
\end{equation}
is also a quasi-isomorphism. Note that it is still induced by $(-1)^{d(d-1)/2}
S$.

As we observed before, the complex $\shF^{\bullet} \tensor_{\shS} \shSr$ on the
left-hand side is a resolution of $\shG$ by basic graded $\shS$-modules. Thus the complex
\[
	D \Bigl( \shHom_{\shS} \bigl( \shF^{\bullet} \tensor_{\shS} \shSr, 
		\omegaS \decal{d} \bigr) \Bigr)
\]
is quasi-isomorphic to $\fl \shH_X^0(\shG_E) \simeq \shG$ by local duality, as in
\subsecref{subsec:loc-dual-vb}. On the other hand, it computes the
$\shS$-module $D \bigl( \shGd \bigr)$, and the isomorphism $D \bigl( \shGd \bigr) \simeq
\shG$ in Proposition~\ref{prop:local-duality} is therefore directly given by that complex.

Returning to \eqref{eq:pol-quasi}, we find that
\[
	D \Bigl( \shHom_{\shS} \bigl( \shF^{\bullet} \tensor_{\shS} \shSr, 
		\omegaS \decal{d} \bigr) \Bigr) \simeq
	D \bigl( \shF^{\bullet}(-n) \bigr) \simeq D \bigl( \shGr(-n) \bigr),
\]
since $\shF^{\bullet}$ resolves $\shGr$. If we compose this isomorphism with the
inverse of $D \bigl( \shGd \bigr) \simeq \shG$, we obtain a map
\[
	\shG \to D \bigl( \shGr(-n) \bigr);
\]
by construction, it is the inverse of the isomorphism in \eqref{eq:HM-seq}. On the
other hand, our derivation shows that it is still given by Saito's formula
$(-1)^{d(d-1)/2} S$, and so the remaining assertion of Lemma~\ref{lem:polarization} is proved.

\section{Several consequences}

\subsection{Derived category formulation}
\label{subsec:derived}

When applying other functors, it is more convenient to have a version of \eqref{eq:HM-seq} in the
derived category; such a version is easily deduced from
\subsecsref{subsec:loc-coh-vb}{subsec:CM}.

Throughout, we will employ the following (mostly standard) notation for derived
categories and derived functors. We let $\Dbqc(\OX)$ be the bounded derived category of
quasi-coherent sheaves on $X$, and $\Db(\OX)$ the full subcategory of objects whose
cohomology sheaves are coherent. Similarly, we write $\Dbqcgr(\shS)$ for the bounded
derived category of quasi-coherent and graded $\shS$-modules, and $\Dbgr(\shS)$ for
the full subcategory of objects whose cohomology sheaves are finitely generated as
$\shS$-modules. As already mentioned, there is an equivalence of categories
\[
	\derR \fl \colon \Dbqc(\shOE) \to \Dbqc(\shS);
\]
for an object $F \in \Dbqc(\shS)$, we write $F_E$ for the corresponding object in
$\Dbqc(\shOE)$, so that $\derR \fl F_E \simeq F$. Similarly, for $G \in
\Dbqcgr(\shS)$, we let $G_P$ be the corresponding object in $\Dbqc(\OP)$. Both
operations are exact functors.

The symbol $\shH^i(\argbl)$ means the $i$-th cohomology sheaf of a complex of
sheaves. The derived functor of the tensor product will be denoted by $\tensorL$.
We write
\[
	\derL D \colon \Dbqcgr(\shS) \to \Dbqcgr(\shS) \op
\]
for the derived functor of the graded dual in \eqref{eq:graded-dual}; we also let
\[
	\derR \Gamma_{\ast} \colon \Db(\OP) \to \Dbgr(\shS)
\]
be the derived functor of $\shF \mapsto \Gammast(\shF) = \bigoplus_{k \in \ZZ} \shF \tensor \OP(k)$.

The results about local cohomology and local duality from
\subsecsref{subsec:loc-coh-vb}{subsec:loc-dual-vb} are easily translated into the
language of derived categories. To begin with, we have the following restatement of
Proposition~\ref{prop:local-cohomology}.

\begin{lemma} \label{lem:der1}
For any object $G \in \Dbgr(\shS)$, there is a functorial exact triangle
\ltriangle{\derR \fl \derR \shH_X^0(G_E)}{G}%
	{\derR \Gammast(G_P)}
in the derived category $\Dbqcgr(\shS)$ of graded, quasi-coherent $\shS$-modules on
$X$.
\end{lemma}

\begin{proof}
The result from \cite{SGA2}*{p.~16} that was used in the proof of
Proposition~\ref{prop:local-cohomology} is based on the exact triangle
\ltriangle{\derR \shH_X^0(G_E)}{G_E}%
		{\derR \jl \ju G_E}
in the derived category $\Dbqc(\shOE)$ of quasi-coherent sheaves on $E$. To get the
conclusion, simply apply the functor $\derR \fl$ to that triangle, and then argue as
before.
\end{proof}

In like manner, local duality from Proposition~\ref{prop:local-duality} can be
reformulated as follows. Note that this only works for finitely generated
$\shS$-modules, because of the necessity of resolving by basic $\shS$-modules.

\begin{lemma} \label{lem:der2}
For any object $G \in \Dbgr(\shS)$, there is a functorial isomorphism
\[
	\derR \fl \derR \shH_X^0(G_E) \simeq 
		\derL D \, \derR \shHom_{\shS} \bigl( G, \omegaS \decal{d} \bigr)
\]
in $\Dbqcgr(\shS)$.
\end{lemma}

By combining the two lemmas with the arguments from
\subsecsref{subsec:loc-dual-CM}{subsec:CM}, we arrive at the following derived-category
version of Theorem~\ref{thm:HM}.

\begin{proposition} \label{prop:derived}
With the assumptions and the notation of Theorem~\ref{thm:HM}, there is an exact triangle
\ltriangle{\derL D \bigl( \shGr(-n) \bigr)}{\shG}%
		{\derR \Gammast \bigl( \shC(\Mmod, F) \bigr)}
in $\Dbqcgr(\shS)$.
\end{proposition}

\subsection{Graded de Rham complexes}
\label{subsec:graded-DR}

We now wish to apply Proposition~\ref{prop:derived} to the study of the de Rham
complex $\DR_X(\Mmod)$ of a filtered $\Dmod$-module $(\Mmod, F)$ underlying a
polarized Hodge module. Since our result gives information about the associated
graded complex, we begin by proving several simple lemmas about the de Rham complex
for general graded $\shS$-modules.

Given any graded $\shS$-module $\shF$, we can form the Koszul complex for the
multiplication map $\Theta_X \tensor_{\OX} \shF \to \shF(1)$, and tensor by
$\omegaX$, to arrive at the Koszul-type complex
\begin{diagram}
	\DR(\shF) = \Bigl\lbrack
		\shF &\rTo& \OmX{1} \tensor \shF(1) &\rTo& \OmX{2} \tensor \shF(2) 
			&\rToDots& \OmX{d} \tensor \shF(d)
	\Bigr\rbrack \decal{d}
\end{diagram}
of graded $\shS$-modules. We call this the \define{de Rham complex} for the graded
module $\shF$; as commonly done, we put it in degrees $-d, \dotsc, 0$,
as indicated by the shift.  Obviously, we have $\DR(\shF) \simeq \DR(\shS)
\tensor_{\shS} \shF$. Since $\shS = \Sym \Theta_X$, the complex $\DR(\shS)$ is a
free resolution of $\omegaX(d) = \omegaX \tensor_{\shS} \shS(d)$, where $\omegaX$
has the trivial $\shS$-module structure; we therefore have
$\DR(\shS) \simeq \omegaX(d)$ in $\Dbgr(\shS)$. It follows that
$\shF \mapsto \DR(\shF) \simeq \DR(\shS) \tensor_{\shS} \shF \simeq \omegaX
\tensorL_{\shS} \shF(d)$ gives rise to an exact functor 
\[
	\DR \colon \Dbqcgr(\shS) \to \Dbqcgr(\shS).
\]

A simple, but useful observation is that the cohomology sheaves of the de Rham
complex are always of finite length as $\shS$-modules.

\begin{lemma} \label{lem:der3}
Let $G \in \Dbqcgr(\shS)$ be any object. Then $\DR(G)_P \simeq 0$ in $\Dbqc(\OP)$.
\end{lemma}

\begin{proof}
As a $\shS$-module, $\omegaX(d)$ is torsion, and so its associated coherent sheaf
on $P$ is zero. Thus we also have $\DR(\shS)_P \simeq 0$. Alternatively,
one can consider the Euler sequence
\begin{diagram}
0 &\rTo& \OP &\rTo& \piu \OmX{1} \tensor \OP(1) &\rTo& \Theta_{P/X} &\rTo& 0
\end{diagram}
on the projectivized cotangent bundle $\pi \colon P \to X$, and observe that its
$d$-th wedge product gives an exact complex
\begin{diagram}
	\Bigl\lbrack 
		\OP &\rTo& \piu \OmX{1} \tensor \OP(1) &\rTo& \piu \OmX{2} \tensor \OP(2)
		&\rToDots& \piu \OmX{d} \tensor \OP(d)
	\Bigr\rbrack.
\end{diagram}
It quickly follows that $\DR(\shS)_P$ is also exact, and therefore isomorphic to zero in
$\Db(\OP)$. Either way, we then have
\[
	\DR(G)_P \simeq \Bigl( \DR(\shS) \tensorL_{\shS} G \Bigr)_P \simeq
		\DR(\shS)_P \tensorL_{\OP} G_P \simeq 0,
\]
because the operation $(\argbl)_P$ is compatible with tensor products. 
\end{proof}

The results of \subsecref{subsec:derived} take a very simple form when
applied to a de Rham complex.

\begin{lemma} \label{lem:der4}
For any object $G \in \Dbgr(\shS)$, there are functorial isomorphisms
\[
	\DR(\shS) \tensorL_{\shS} \derL D \, 
		\derR \shHom_{\shS} \bigl( G, \omegaS \decal{d} \bigr) \simeq
	\derL D \, \derR \shHom_{\shS} \bigl( \DR(G), \omegaS \decal{d} \bigr)
		\simeq \DR(G).
\]
\end{lemma}

\begin{proof}
Apply Lemma~\ref{lem:der1} to the object $\DR(G)$ to obtain an exact triangle
\ltriangle{\derR \fl \derR \shH_X^0 \bigl( \DR(G)_E \bigr)}{\DR(G)}%
	{\derR \Gammast \bigl( \DR(G)_P \bigr)}
whose third term is isomorphic to zero by Lemma~\ref{lem:der3}. The triangle
therefore degenerates to an isomorphism
\begin{equation} \label{eq:der4-iso}
	\derR \fl \derR \shH_X^0 \bigl( \DR(G)_E \bigr) \simeq \DR(G).
\end{equation}
Lemma~\ref{lem:der2} now implies one half of the assertion. For the other, we note that
\begin{align*}
	\derR \fl \derR \shH_X^0 \bigl( \DR(G)_E \bigr) 
		&\simeq \derR \fl \derR \shH_X^0 \bigl( \DR(\shS)_E \tensorL_{\shOE} G_E \bigr) \\
		&\simeq \DR(\shS) \tensorL_{\shS} \derR \fl \derR \shH_X^0(G_E),
\end{align*}
and then conclude by invoking Lemma~\ref{lem:der2} a second time.
\end{proof}

\subsection{Properties of the de Rham complex}
\label{subsec:deRham}

After these preparations, we can now apply Theorem~\ref{thm:HM} to study the de Rham complex
\begin{diagram}
	\DR_X(\Mmod) = \Bigl\lbrack
		\Mmod &\rTo& \OmX{1} \tensor \Mmod &\rTo& \OmX{2} \tensor \Mmod &\rToDots&
		\OmX{d} \tensor \Mmod
	\Bigr\rbrack \decal{d}
\end{diagram}
of the filtered $\Dmod$-module $(\Mmod, F)$. As in \subsecref{subsec:graded-DR}, the
complex is supported in degrees $-d, \dotsc, 0$. It is naturally filtered by subcomplexes
\begin{diagram}
	F_k \DR_X(\Mmod) = \Bigl\lbrack
		F_k \Mmod &\rTo& \OmX{1} \tensor F_{k+1} \Mmod &\rToDots&
		\OmX{d} \tensor F_{k+n} \Mmod
	\Bigr\rbrack \decal{d};
\end{diagram}
moreover, $\Gr^F \DR_X(\Mmod)$ is a complex of finitely-generated graded
$\shS$-modules. 

As a matter of fact, it is not hard to see that
\[
	\Gr^F \DR_X(\Mmod) \simeq \DR \bigl( \Gr^F \! \Mmod \bigr) = \DR(\shG).
\]
We can therefore obtain information about the associated graded of the de Rham
complex of $\Mmod$ by applying Lemma~\ref{lem:der4} to the complex $\DR(\shG)$. When
combined with Proposition~\ref{prop:derived}, the following result emerges.

\begin{proposition} \label{prop:deRham}
Let $M$ be a polarized Hodge module of weight $w = d + n$ on a nonsingular
$d$-dimensional complex algebraic variety, let $(\Mmod, F)$ be the underlying
filtered left $\Dmod$-module, and $\shG = \Gr^F \! \Mmod$ the associated graded
$\shS$-module. Let $\DR(\shG) = \Gr^F \DR_X(\Mmod)$ be the graded de Rham complex.
Then 
\[
	\DR(\shS) \tensorL_{\shS} \derL D \bigl( \shGr(-n) \bigr) \simeq \DR(\shG)
\]
in the derived category $\Dbgr(\shS)$ of graded, coherent $\shS$-modules.
\end{proposition}

In particular, there is a convergent spectral sequence
\[
	E_1^{p,q} = \bigl( \OmX{p+d} \tensor_{\OX} \shS(p+d) \bigr) \tensor_{\shS} D^q \bigl( \shGr(-n) \bigr)
		\Longrightarrow \shH^{p+q} \bigl( \Gr^F \DR_X(\Mmod) \bigr)
\]
of graded $\shS$-modules; explicitly, the degree $k$ part is
\[
	\bigl( E_1^{p,q} \bigr)_k = 
		\shExt_{\OX}^q \bigl( \shG_{-w-k-p}, \OmX{p+d} \bigr)
		\Longrightarrow \shH^{p+q} \bigl( \Gr_k^F \DR_X(\Mmod) \bigr).
\]
The spectral sequence has a useful consequence for the de Rham complex of $\Mmod$; of
course, this type of result can also be proved directly from Saito's theory.
\begin{proposition}
Let $m \in \ZZ$ be such that $F_{-m} \Mmod = 0$. Then the inclusion
$F_{m-n-1} \DR_X(\Mmod) \subseteq \DR_X(\Mmod)$ is a filtered quasi-isomorphism.
\end{proposition}
\begin{proof}
By assumption, $\shG_k = 0$ for all $k \leq -m$; in the spectral sequence, we
therefore have $\bigl( E_1^{p,q} \bigr)_k = 0$ for all $p, q \in \ZZ$, provided that
$-w-k+d = -n-k \leq -m$. This means that $\Gr_k^F \DR_X(\Mmod)$ is exact for $k \geq m-n$. 
Since $F_{\bullet} \Mmod$ is a good filtration on $\Mmod$, the assertion follows.
\end{proof}

To see what this means, let us suppose that $M$ is the intermediate extension of a
polarized variation of Hodge structure of weight $n$, which is ``geometric,''
meaning such that $F^{n+1} \Mmod = F_{-n-1} \Mmod = 0$. Then $F_0 \DR_X(\Mmod)
\subseteq \DR_X(\Mmod)$ is a filtered quasi-isomorphism by the proposition. This
implies, for instance, that any class in $\HH^{-d+1} \bigl( \DR_X(\Mmod) \bigr)$ can
be represented by a closed form $\omega \in \Gamma \bigl( X, \OmX{1} \tensor F_1
\Mmod \bigr)$. When applied to the $\Dmod$-module $\Mvan$ from
\subsecref{subsec:background}, this is useful in the study of normal functions as in
\cite{Schnell}.

\section{Notation used in the paper}

\subsection{Graded modules}

For $E \to X$ a vector bundle, and $\shE$ the corresponding locally free sheaf on
$X$, we set $\shS = \Sym \shEdual$. We define
\[
	\omegaS = \omegaX^{-1} \tensor_{\OX} \shS(-d).
\]
Let $\shG$ be a graded $\shS$-module. Its graded dual is the module
\[
	D(\shG) = \bigoplus_{k \in \ZZ} \shHom_{\OX} \bigl( \shG_{-k}, \OX \bigr),
\]
with $\shHom \bigl( \shG_{-k}, \OX \bigr)$ in degree $k$, and the module structure
given by the rule $(s \phi)(g) = \phi(s g)$. The involution $(-1) \colon E \to E$ allows us to define
\[
	\shGr = (-1)^{\ast} \shG;
\]
this is the same graded $\OX$-module as $\shG$, but with the $\shS$-module structure
changed so that sections of $\shS_k$ act with an extra factor of $(-1)^k$. When
$\shG$ is Cohen-Macaulay as an $\shS$-module, we let
\[
	\shGd = \shExt_{\shS}^d \bigl( \shG, \omegaS \bigr)
\]
be the dual $\shS$-module.


\end{document}